\documentclass[11pt,reqno]{amsart}
\usepackage{amsmath,amssymb,amsthm,fullpage, graphicx, chngcntr,float,mathtools,hyperref}
\allowdisplaybreaks
\newtheorem{theorem}{Theorem}[section]
\newtheorem{lemma}{Lemma}[section]

\theoremstyle{definition}
\newtheorem{definition}{Definition}[section]

\newcommand{\hpqTUPLE}{\mathcal{H}^*_{p,q}}

\newcommand{\N}{\mathbb{N}}

\newcommand{\Cn}{\mathbb{C}^n}
\newcommand{\Sn}{\mathbb{S}^{2n-1}}

\newcommand{\3}{\mathbb{S}^3}

\newcommand{\dbarb}{\overline{\partial}_b}
\newcommand{\dbarbstar}{\overline{\partial}_b^{\ast}}

\newcommand{\conj}[1]{\overline{#1}}

\renewcommand{\l}{\mathcal{L}}
\newcommand{\lbar}{\smash[b]{\conj{\mathcal{L}}}}
\newcommand{\boxb}{\square_b}
\newcommand{\boxbt}{\square_b^t}
\newcommand{\restrict}{\raise-.3ex\hbox{\ensuremath|}}

\counterwithin{figure}{section}

\newcommand{\HH}{\mathcal{H}}

\newcommand{\LL}[1]{\mathcal{L}_{#1}}

\newcommand{\lt}{\mathcal{L}_t}

\newcommand{\hpq}{\HH_{p,q}}

\newcommand{\dbb}{\overline{\partial}_b} 
\newcommand{\dbba}{\overline{\partial}_b^*} 

\newcommand\CC{\mathbb C}

\newcommand{\heart}{\mathcal{M}}

\title{Spectra of Kohn Laplacians on Spheres}

\author{John Ahn}
\address[John Ahn]{Bowdoin College, 7 Smith Union, Brunswick, ME 04011, USA}
\email{jtahn@bowdoin.edu}

\author{Mohit Bansil}
\address[Mohit Bansil]{Michigan State University, 619 Red Cedar Road C212 Wells Hall, East Lansing, MI 48824, USA}
\email{bansilmo@msu.edu}

\author{Garrett Brown}
\address[Garrett Brown]{Harvard University, 153 Kirkland Mail Center, 95 Dunster Street, Cambridge, MA 02138, USA}
\email{garrettbrown@college.harvard.edu}

\author{Emilee Cardin}
\address[Emilee Cardin]{College of William and Mary, CSU 5358, P.O. Box 8793, Williamsburg, VA 23187, USA}
\email{elcardin@email.wm.edu}

\author{Yunus E. Zeytuncu}
\address[Yunus E. Zeytuncu]{University of Michigan--Dearborn, Department of Mathematics and Statistics, 2048 Evergreen Road, Dearborn, MI 48128, USA}
\email{zeytuncu@umich.edu}

\thanks{This work is supported by NSF (DMS-1659203). The work of the fourth author is also partially supported by a grant from the Simons Foundation (\#353525).}

\begin{document}
\maketitle
\begin{abstract}
In this note, we study the spectrum of the Kohn Laplacian on the unit spheres in $\mathbb{C}^n$ and revisit Folland's classical eigenvalue computation. We also look at the growth rate of the eigenvalue counting function in this context. Finally, we consider the growth rate of the eigenvalues of the perturbed Kohn Laplacian on the Rossi sphere in $\mathbb{C}^2$. 
\end{abstract}


\section{Introduction}
\subsection{Background}
The unit sphere $\Sn\subset \Cn$ is a CR manifold (of hypersurface type) with the CR structure induced from the ambient space. By following the standard setting we define the tangential Cauchy-Riemann complex with the operators $\dbarb$ and $\dbarbstar$ on the spaces of square integrable $(0,q)$-forms $L^2_{(0,q)}(\Sn)$\footnote{For simplicity we restrict our attention to $(0,q)$ forms instead of $(p,q)$ forms.}. The Kohn Laplacian (or $\dbarb$-Laplacian) 
\begin{equation*}
\boxb=\dbarb\dbarbstar+\dbarbstar\dbarb
\end{equation*}
is a linear, closed, densely defined self-adjoint operator from $L^2_{(0,q)}(\Sn)$ to itself. The analytic properties of this second order differential operator are closely related to the geometry of the underlying manifold (although we work here on $\Sn$ the same setup works on other CR manifolds). We refer the reader to \cite[Chapter 7]{CS01} for the details of this setup.
\subsection{Spherical Harmonics} In this part we list definitions and theorems that are needed in the rest of the paper. For a detailed study of spherical harmonics we refer the reader to \cite{Axler13Harmonic}.

We say a complex polynomial $p(z)$ is homogeneous of degree $k$ if $p(\lambda z)=\lambda^kp(z)$ for all $z\not=0$. Similarly, $p(z,\overline{z})$ is called homogeneous of bidegree $(p,q)$ if $f(\lambda_1 z, \lambda_2 \overline{z})= \lambda_1^p\lambda_2^qp(z,\overline{z})$ for all $z\not=0$. 
We say a twice-differentiable function $f$ is harmonic if $\triangle f=0$ , where the Laplacian is defined by $$\triangle f = 4\sum_{i=1}^n \frac{\partial^2 f}{\partial z_i\partial\overline{z_i}}.$$ 
A spherical harmonic is the restriction to $\Sn$ of a complex polynomial that is harmonic on $\Cn$. 
We use $\mathcal{H}_k(\CC^n)$ to denote the space of all harmonic, homogeneous polynomials of degree $k$ on $\Cn$ and 
$\mathcal{H}_{p,q}(\Cn)$ for the space of all harmonic, homogeneous polynomials of bidegree $(p,q)$. 
Similarly we use $\mathcal{H}_{k}(\Sn)$ and $\mathcal{H}_{p,q}(\Sn)$ to denote the restrictions of these spaces on $\Sn$. The following decomposition theorem is fundamental in our study of $\boxb$ on $L^2(\Sn)$.
\begin{theorem}\cite[Theorem 3.7]{Klima}
	The spaces $\hpq (\Sn)$ are pairwise orthogonal, and 
	\[ L^2 (\Sn) = \bigoplus_{p,q=0}^\infty \hpq(\Sn). \]
\end{theorem}
By using a standard counting argument one obtains the following formula for the dimensions of the spaces of spherical harmonics.
\begin{lemma}\cite[Corollary 3.10]{Klima}\label{dimension}
	For $p,q \geq 1$,
	\begin{align*} \dim({\mathcal{H}_{p,q}}(\Sn))
	&= \binom{n + p - 1}{p} \binom{n + q - 1}{q}  - \binom{n + p - 2}{p - 1} \binom{n + q - 2}{q - 1}\\
	&=\frac{(n-1)(n+p+q-1)}{pq} \binom{n+p-2}{p-1} \binom{n+q-2}{q-1}.
	\end{align*}
\end{lemma}
\vskip 1cm

\subsection{Notation} In the rest of the note we use the standard $\Omega$ and $O$ notation to denote asymptotic lower and upper bounds, respectively. That is, given two functions $f$ and $g$, we say $f = \Omega(g)$ if there exists a constant $c>0$ such that $f(x) \geq c g(x)$ as $x \rightarrow \infty$. Similarly, $f = O(g)$ if there exists $c>0$ such that $f(x) \leq c g(x)$ as $x \rightarrow \infty$. Finally, we say $f = \Theta(g)$ if $f = \Omega(g)$ and $f = O(g)$.

\subsection{Results}

In \cite{Folland}, Folland computes the eigenvalues and eigenforms of $\boxb$ on $L^2_{(0,q)}(\Sn)$ by using unitary representations. 
\begin{theorem}\label{Fol}
	$\hpq \left(\Sn\right)$ is an eigenspace for $\dbba \dbb$ with the associated eigenvalue $2q(p+n-1)$.
\end{theorem}
In the second section of this note we go over these computations on the space of square integrable functions (i.e. $L^2(\Sn)$) by using spherical harmonics and present eigenvalue computations in an accessible way. This more elementary approach enables us to write a code\footnote{The code can be downloaded at \texttt{https://goo.gl/kBsUzA}.} in \texttt{SymPy} that computes the eigenvalues of $\boxb$ and other similar second order differential operators defined on $L^2(\Sn)$. Furthermore, by using the explicit forms of the eigenvalues and formulas for the dimensions of spherical harmonic subspaces of $L^2(\Sn)$, we study the growth rate for the counting function of the eigenvalues. For $m\in\mathbb{Z}$, let $N(m)$ be the number of eigenvalues of $\boxb$ on $L^2(\Sn)$ that are less or equal than $m$ with counting multiplicity.

\begin{theorem}{\label{countB}}
There exists a real $c > 0$ so that $ \frac{1}{c} m^n \leq N(m) \leq c m^n$, i.e. $N(m)=\Theta(m^n).$
\end{theorem}

In other words, here we prove that $$\limsup_{m\to\infty}\frac{N(m)}{m^n}\in (0,\infty).$$ It would be interesting to compute the exact limit and check if it is related to the surface area of $\Sn$. Indeed, in the case of the Laplace-Beltrami operator Weyl's law states that this ratio is the surface area of $\Sn$.

In addition to the induced CR structure from the ambient manifold, one can define different intrinsic CR structures on a given manifold, see \cite[Chapter 8]{Boggess91CR}. The most famous example of these abstract CR manifolds is the Rossi sphere. It is known that the Rossi sphere is not globally CR embeddable into any $\Cn$ \cite{Burns79}. This can be seen by explicitly studying the perturbed Kohn Laplacian (defined by the abstract CR structure) and looking at its essential spectrum. In \cite{REU17} authors studied the bottom of the spectrum of the perturbed Kohn Laplacian by using spherical harmonics. In the last section of this note we continue this study and provide the growth rate of the largest eigenvalues from each subspace of spherical harmonics.



\section{Eigenvalues of $\boxb$ on $L^2(\Sn)$}

\subsection{Explicit Eigenvalue Computation} 

In this section we study the eigenvalues of $\boxb$ on $L^2(\Sn)$. Since $\dbarbstar$ is identically zero on $L^2(\Sn)$, $\boxb$ simplifies on $L^2(\Sn)$ as
\begin{equation*}
\boxb=\dbarbstar\dbarb.
\end{equation*}
Before we compute the eigenvalues we present the operators $\dbarb$ and $\dbarbstar$ in coordinate forms. A smooth differential $1$-form $\omega$ on $\Sn$ can be expressed as
\[\omega= \sum_{k =1}^n ( A_k dz_k + B_k d \overline{z}_k ) = A_1  dz_1 +  B_1 d \overline{z}_1 + \cdots + A_n dz_n + B_n d \overline{z}_n \]
where $A_k, B_k \in \mathcal{C}^\infty(\CC^n)$. 
As computed in \cite{Folland}, for a smooth function $f$ on $\Sn$ we have
\[ \dbb f = \sum_{i=1}^n \Big( \frac{\partial f}{\partial \overline{z}_i} - z_i \sum_{a=1}^n \overline{z}_a \frac{\partial f}{\partial \overline{z}_a}   \Big) d\overline{z}_i.  \]
Furthermore, following the normalization of inner products as in \cite{Folland} we have
\begin{align*}
\langle  d\overline{z}_i,  d\overline{z}_j \rangle &= 2 \delta_{ij}\\
\langle dz_i, d \overline{z}_j \rangle &= 0.
\end{align*} 
Using integration by parts, we obtain the following expression for the adjoint operator.

\begin{lemma}{\label{actionDBBA}}
For a smooth $1$-form $\omega=\sum_{k =1}^n ( A_k dz_k + B_k d \overline{z}_k )$,
\[  \dbba \omega=  -2 \sum_{i=1}^n \Big( \frac{\partial}{\partial z_i} B_i - \sum_{a=1}^n \frac{\partial}{\partial z_a} z_a \overline{z}_i B_i \Big). \]
\end{lemma}

\begin{proof}
Let $g$ be a smooth function on $\Sn$. Since we are working on a compact manifold, we don't get any boundary terms when we integrate by parts.
\begin{align*}
    \Big\langle \dbba & \Big( \sum_{k =1}^n ( A_k dz_k + B_k d \overline{z}_k ) \Big), g \Big\rangle \\
    &= \Big\langle  \sum_{k =1}^n ( A_k dz_k + B_k d \overline{z}_k ) , \dbb g \Big\rangle \\
    &= \Big\langle \sum_{k =1}^n  A_k dz_k + \sum_{k =1}^n  B_k d \overline{z}_k , \sum_{i=1}^n \Big( \frac{\partial g}{\partial \overline{z}_i} - z_i \sum_{a=1}^n \overline{z}_a \frac{\partial g}{\partial \overline{z}_a}   \Big) d\overline{z}_i  \Big\rangle \\
    &= 2 \sum_{i=1}^n \Big\langle B_i ,  \frac{\partial g}{\partial \overline{z}_i} - z_i \sum_{a=1}^n \overline{z}_a \frac{\partial g}{\partial \overline{z}_a} \Big\rangle \\
    &= 2 \sum_{i=1}^n \Big( \Big\langle B_i ,  \frac{\partial g}{\partial \overline{z}_i}  \Big\rangle - \sum_{a=1}^n \Big\langle B_i , z_i \overline{z}_a \frac{\partial g}{\partial \overline{z}_a}  \Big\rangle \Big)\\
    &= 2 \sum_{i=1}^n \Big( -\Big\langle \frac{\partial}{\partial z_i} B_i ,  g \Big\rangle + \sum_{a=1}^n \Big\langle  \frac{\partial}{\partial z_a} z_a \overline{z}_i B_i , g   \Big\rangle \Big)\\
    &= -2 \sum_{i=1}^n \Big( \Big\langle \frac{\partial}{\partial z_i} B_i ,  g \Big\rangle - \sum_{a=1}^n \Big\langle  \frac{\partial}{\partial z_a} z_a \overline{z}_i B_i , g   \Big\rangle \Big)\\
    &= -2 \sum_{i=1}^n \Big\langle \frac{\partial}{\partial z_i} B_i - \sum_{a=1}^n \frac{\partial}{\partial z_a} z_a \overline{z}_i B_i , g   \Big\rangle \\
    &= \Big\langle -2 \sum_{i=1}^n \Big( \frac{\partial}{\partial z_i} B_i - \sum_{a=1}^n \frac{\partial}{\partial z_a} z_a \overline{z}_i B_i \Big), g   \Big\rangle
\end{align*}
By comparing the beginning and end of the identity we prove the lemma.
\end{proof}

Before we look at the action of $\boxb$ on a square integrable function we look at the action of two other operations on the spherical harmonics.

\begin{lemma}
If $f \in \hpq (\Sn)$, then
\[ \sum_{k=1}^n z_k \frac{\partial f}{\partial z_k} = pf \qquad \qquad \textrm{and} \qquad \qquad \sum_{k=1}^n \overline{z}_k \frac{\partial f}{\partial \overline{z}_k}  = qf. \]
\end{lemma}
\begin{proof}
Consider a polynomial $f \in \hpq$. So $f$ is harmonic homogeneous of bidegree $p, q$. Then for each monomial term $g=z_1^{\alpha_1}\cdots z_n^{\alpha_n}\overline{z}_1^{\beta_1}\cdots \overline{z}_n^{\beta_n}$ of $f$, we have:
\[ \sum_{k=1}^n z_k \frac{\partial g}{\partial z_k} = \sum_{k=1}^n (\alpha_k ) g  = \Big( \sum_{k=1}^n \alpha_k \Big) g = p g \]

\[ \sum_{k=1}^n \overline{z}_k \frac{\partial g}{\partial \overline{z}_k} = \sum_{k=1}^n (\beta_k ) g  = \Big( \sum_{k=1}^n \beta_k \Big) g = q g. \]
So each monomial term $g$ is scaled by $p$ (resp. $q$). By the linearity of differential operators, $f$ is scaled by $p$ (resp. $q$) as well.  
\end{proof}

By combining the lemmas above we obtain the eigenvalues of $\boxb$.

\begingroup
\def\thetheorem{\ref{Fol}}
\begin{theorem}
	$\hpq \left(\Sn\right)$ is an eigenspace for $\dbba \dbb$ with the associated eigenvalue $2q(p+n-1)$.
\end{theorem}
\addtocounter{theorem}{-1}
\endgroup

\begin{proof}
For $f \in \hpq(\Sn)$:
\begin{align*}
\dbba \dbb f 
&= \dbba \Bigr[  \sum_{i=1}^n \Big( \frac{\partial f}{\partial \overline{z}_i} - z_i \sum_{a=1}^n \overline{z}_a \frac{\partial f}{\partial \overline{z}_a}   \Big)  d\overline{z}_i \Bigr] \\
&= \dbba \Bigr[  \sum_{i=1}^n \Big( \frac{\partial f}{\partial \overline{z}_i} - z_i q f   \Big)  d\overline{z}_i \Bigr]  \\
&= -2 \sum_{i=1}^n \Bigr[ \frac{\partial}{\partial z_i} \Big( \frac{\partial f}{\partial \overline{z}_i} - z_i q f   \Big) - \sum_{a=1}^n \frac{\partial }{\partial z_a} z_a \overline{z}_i \Big( \frac{\partial f}{\partial \overline{z}_i} - z_i q f   \Big) \Bigr] \\
&= -2 \sum_{i=1}^n \Bigr[ \Big( \frac{\partial^2 f}{\partial z_i \partial \overline{z}_i} - \frac{\partial}{\partial z_i} z_i q f   \Big) - \sum_{a=1}^n \Big( \frac{\partial }{\partial z_a} z_a \overline{z}_i \frac{\partial f}{\partial \overline{z}_i} - \frac{\partial }{\partial z_a} z_a \overline{z}_i  z_i q f   \Big) \Bigr] \\
&= -2 \sum_{i=1}^n \frac{\partial^2 f}{\partial z_i \partial \overline{z}_i} + 
2 \sum_{i=1}^n \frac{\partial}{\partial z_i} z_i q f + 
2 \sum_{i=1}^n \sum_{a=1}^n \frac{\partial}{\partial z_a} z_a \overline{z}_i \frac{\partial f}{\partial \overline{z}_i} 
- 2 \sum_{i=1}^n \sum_{a=1}^n \frac{\partial}{\partial z_a} z_a \overline{z}_i  z_i q f. 
\end{align*}
We start the first term. Because $f$ is harmonic, we know $0 = \triangle(f) = 4\sum_{i=1}^n  \frac{\partial^2 f}{\partial z_i \partial \overline{z}_i}$. Thus, we have
\[ 0 = \sum_{i=1}^n  \frac{\partial^2 f}{\partial z_i \partial \overline{z}_i} = -2 \sum_{i=1}^n \frac{\partial^2 f}{\partial z_i \partial \overline{z}_i}. \]
Now, for the second and third terms, we apply the product rule.

\begin{align*}
2 \sum_{i=1}^n \frac{\partial}{\partial z_i} z_i q f &= 2q \sum_{i=1}^n \frac{\partial}{\partial z_i} z_i f \\
&= 2q \sum_{i=1}^n \Big( z_i \frac{\partial f}{\partial z_i} + f \Big) \\
&= 2q \Bigr[ \sum_{i=1}^n  z_i \frac{\partial f}{\partial z_i} + \sum_{i=1}^n f  \Bigr] \\
&= 2q (p + n )f
\end{align*}

\begin{align*}
2 \sum_{i=1}^n \sum_{a=1}^n \frac{\partial}{\partial z_a} z_a \overline{z}_i \frac{\partial f}{\partial \overline{z}_i}  &= 2  \sum_{a=1}^n \frac{\partial}{\partial z_a} z_a  \sum_{i=1}^n \overline{z}_i \frac{\partial f}{\partial \overline{z}_i}  \\
&= 2  \sum_{a=1}^n \frac{\partial}{\partial z_a} z_a  q f \\
&= 2q \sum_{a=1}^n \Big( z_a \frac{\partial f}{\partial z_a} + f \Big)  \\
&= 2q (p + n )f 
\end{align*}

Now recall that on $\Sn$, we have $z_1 \overline{z}_1 + \cdots + z_n \overline{z}_n = 1$. Thus,

\[ \sum_{a=1}^n \sum_{i=1}^n z_i \overline{z}_i f  = \sum_{a=1}^n f = nf.  \]

We also go over the following explicit computation (again by using linearity we can assume $f$ is a monomial and $f=z_1^{\alpha_1}\cdots z_n^{\alpha_n}\overline{z}_1^{\beta_1}\cdots \overline{z}_n^{\beta_n}$):
\begin{align*}
\sum_{a=1}^n z_a \frac{\partial}{\partial z_a} \sum_{i=1}^n z_i \overline{z}_i f 
&= \sum_{a=1}^n z_a \frac{\partial}{\partial z_a} \Big( z_1 \overline{z}_1 + \cdots + z_n \overline{z}_n \Big) f \\
&= \sum_{a=1}^n z_a \Big( \frac{\partial}{\partial z_a} z_1 \overline{z}_1 f + \cdots + \frac{\partial}{\partial z_a} z_a \overline{z}_a f + \cdots + \frac{\partial}{\partial z_a} z_n \overline{z}_n  f \Big) \\
&= \sum_{a=1}^n z_a \Big( \frac{\alpha_a}{z_a} z_1 \overline{z}_1 f + \cdots + \frac{\alpha_a + 1}{z_a} z_a \overline{z}_a f + \cdots + \frac{\alpha_a}{z_a} z_n \overline{z}_n f \Big) \\
&= \sum_{a=1}^n \Big( (\alpha_a) z_1 \overline{z}_1 f + \cdots + (\alpha_a + 1) z_a \overline{z}_a f + \cdots + (\alpha_a) z_n \overline{z}_n f \Big) \\
&= \sum_{i=1}^n ( \alpha_1 + \cdots + (\alpha_i + 1) + \cdots + \alpha_n ) z_i \overline{z}_i f \\
&= \sum_{i=1}^n (p+1) z_i \overline{z}_i f  
= (p+1) \sum_{i=1}^n  z_i \overline{z}_i f 
= (p+1)f \\
\end{align*}

We are now ready to compute the fourth term of the $\dbba \dbb f$ expansion: 

\begin{align*}
- 2 \sum_{i=1}^n \sum_{a=1}^n \frac{\partial}{\partial z_a} z_a \overline{z}_i z_i q f 
&= - 2 q \Big( \sum_{a=1}^n \frac{\partial}{\partial z_a} z_a \sum_{i=1}^n z_i \overline{z}_i f  \Big) \\
&= - 2 q \Big( \sum_{a=1}^n \Big( z_a \frac{\partial}{\partial z_a} + I \Big) \sum_{i=1}^n z_i \overline{z}_i f \Big) \\
&= - 2 q \Big( \sum_{a=1}^n z_a \frac{\partial}{\partial z_a} \sum_{i=1}^n z_i \overline{z}_i f +  \sum_{a=1}^n \sum_{i=1}^n z_i \overline{z}_i f \Big) \\
&= - 2 q(p + 1 + n)f
\end{align*}

Returning to our original computation of $\dbba \dbb f$, we now have:

\begin{align*}
\dbba \dbb f &= -2 \sum_{i=1}^n \frac{\partial^2 f}{\partial z_i \partial \overline{z}_i} + 
2 \sum_{i=1}^n \frac{\partial}{\partial z_i} z_i q  + 
2 \sum_{i=1}^n \sum_{a=1}^n \frac{\partial}{\partial z_a} z_a \overline{z}_i \frac{\partial f}{\partial \overline{z}_i} 
- 2 \sum_{i=1}^n \sum_{a=1}^n \frac{\partial}{\partial z_a} z_a \overline{z}_i  z_i q   \\
&= 0 + 2q (p + n )f  + 2q (p + n )f  - 2 q(p + 1 + n)f \\
&= 2q(p + n - 1)f.
\end{align*}
\end{proof}


\subsection{Asymptotics of Counting Function}

In this part we look at the counting function $N(m)$. 
\begin{definition}
For $m\in\mathbb{Z}$, let $N(m)$ be the number of eigenvalues of $\boxb$ on $L^2(\Sn)$ that are less or equal than $m$ with counting multiplicity.
\end{definition}
Similar functions and relations between their asymptotics and geometry of the underlying manifold were studied in \cite{Met, Fu2005,Fu2008}. In particular in some cases the growth rate of $N(m)$ carries information about the type of the manifold \cite{Fu2005,Fu2008}. Furthermore, in the case of the Laplace-Beltrami operator, Weyl's law states that the limit of the ratio $\frac{N(m)}{m^n}$ gives the surface area of $\Sn$. Before we state our result, we recall Lemma \ref{dimension}.

\begingroup
\def\thelemma{\ref{dimension}}
\begin{lemma}\label{dimHPQ}For $p,q \geq 1$,
\[ \dim({\mathcal{H}_{p,q}}(\Sn))=\frac{(n-1)(n+p+q-1)}{pq} \binom{n+p-2}{p-1} \binom{n+q-2}{q-1}. \]
\end{lemma}
\addtocounter{lemma}{-1}
\endgroup

Note that ignoring multiplicity would induce a function with linear growth. Indeed for any even $\widehat{m}$ with $m \geq \widehat{m} > 2(n-1)$, we can solve $\widehat{m} = 2q(p+n-1)$ after fixing $q=1$. Additionally, by convention, we set $N(m) = 0$ when $m < 0$.

We note that when $n=1$, the eigenvalue of $\dbba \dbb = 0$. Indeed, when $n=1$ and when $p$ and $q$ are both nonzero, Lemma $\ref{dimHPQ}$ gives us that the dimension of $\hpq$ is 0. This is because the only harmonic homogeneous polynomials on $\CC$ are of the form $z^p$ or $\overline{z}^q$, which belong to $\HH_{p,0}$ or $\HH_{0,q}$, respectively. Thus, $\hpq$ is nontrivial only when either $p$ or $q$ is zero. However, on such spaces, the eigenvalue of $\dbba \dbb$ on $\hpq$ is 0.

\begin{lemma}{\label{countLB}}
There exists a real constant $c > 0$ so that $ c m^n \leq N(m)$, i.e. $N(m) \in \Omega(m)$. 

\end{lemma}
\begin{proof}
Fix even $m$, then $N(m) - N(m-2)$ is the multiplicity of the eigenvalue $m$, since all the eigenvalues are even by Theorem \ref{Fol}. This requires computing the sum of the dimensions of all $\hpq(\Sn)$ such that the pair $(p,q)$ satisfies the equation $E(p,q) = m$, where $E(p,q) = 2q(p+n-1)$. 
Now assuming $m > 2(n-1)$, there exists a positive integer solution $p = \widehat{p}$ to $E(p,q) =m$ when $q = 1$. Define the solution set $A = \{(p,q) \; | \; E(p,q) = m \}$. Then we have 

\[ N(m) - N(m-2) = \sum_{(p,q) \in A} \dim \hpq \geq \dim \HH_{\widehat{p}, 1}. \]

Note that $\dim \HH_{\widehat{p}, 1} = \Omega (m^{n-1})$, which follows from Lemma \ref{dimension}. Namely, since asymptotically $\widehat{p} = m/2$, we have

\begin{align*}
\dim \HH_{\widehat{p}, 1(\Sn)} 
&= \frac{(n-1)(n+ \widehat{p})}{\widehat{p}} \binom{n+\widehat{p}-2}{n-1} \binom{n-1}{n-1} \\
&\geq  \binom{n+\widehat{p}-2}{n-1} \\
&\geq \frac{1}{(n-1)!} \widehat{p}^{n-1} \\
&= \Omega\Big( \frac{m}{2} \Big)^{n-1} \\
&= \Omega(m^{n-1}).    
\end{align*}
 
Putting it all together, we have that 

\[ 2N(m) \geq N(m) + N(m-1) = \sum_{j=0}^m (N(j) - N(j-2)) \geq \sum_{j= 0}^m \Omega(j^{n-1}) \geq \Omega(m^n).  \]
\end{proof}

\begin{lemma}{\label{countUB}}
There exists a real constant $c > 0$ so that $N(m) \leq c m^n$, i.e. $N(m) = O(m^n)$.
\end{lemma}
\begin{proof}
Again, fix even $m$ and inspect $N(m) - N(m-2)$. Note that asymptotically, we can let our eigenvalue equation be $E(p,q) = 2qp$. Thus, asymptotically we have that 
\[ N(m) - N(m-2) = \sum_{(p,q) \in A} \dim \hpq \lesssim \sum_{(p,q) \in A} (p+q) (pq)^{n-2} =  \sigma(m)O(m^{n-2}), \]
where $\sigma(m)$ is the sum of all divisors of $m$. Thus, we have that 
\[ N(m) \lesssim \sum_{x \leq m} 2x^{n-2} \sigma(x) \lesssim 2m^{n-2} \sum_{x\leq m} \sigma(x) = O(m^n). \]
The last equality follows since $\sum_{x \leq m} \sigma(x) = O(m^2)$. A proof of this fact can be found in Chapter 3.6 of \cite{apostol}.
\end{proof}

By combining the last two lemmas we obtain the following statement.
\begingroup
\def\thetheorem{\ref{countB}}
\begin{theorem}
There exists a real $c > 0$ so that $ \frac{1}{c} m^n \leq N(m) \leq c m^n$, i.e. $N(m)=\Theta(m^n).$
\end{theorem}
\addtocounter{theorem}{-1}
\endgroup

We note that the constants in Lemma \ref{countLB}, Lemma \ref{countUB}, and Theorem \ref{countB} do depend on the dimension $n$. This dependence also agrees with the explicit constants calculated by Weyl for the Laplace-Beltrami operator.

\section{Spectrum of Other Second Order Differential Operators on $L^2(\Sn)$}
Another interesting class of second order differential operators are sum of squares operators $\heart_b$, introduced in the fourth chapter of \cite{Klima}. These operators capture \textit{half} of the action of $\boxb$ on $\3$; in higher dimensions they lead to the study of various possible perturbation of $\boxb$.

We define the sum of squares operator $\heart_b$ on $L^2 ( \Sn)$ as
\[ \heart_b = -(M_{12}\overline{M}_{12} + M_{13} \overline{M}_{13} + \cdots + M_{1n} \overline{M}_{1n}), \]
where $M_{1k} = \overline{z}_1 \frac{\partial}{\partial z_k} - \overline{z}_k \frac{\partial}{\partial z_1}$ and $\overline{M}_{1k} = z_1 \frac{\partial}{\partial \overline{z}_k} - z_k \frac{\partial}{\partial \overline{z}_1}$. Note that one can easily consider $M_{ik}$ for $i\not=1$, for simplicity we focus on the case $i=1$.

For any $f \in \hpq(\Sn)$, the specific degrees of the $z_k, \overline{z}_k$ may vary. For example, both $z_1^2z_2\overline{z}_1^3\overline{z}_2^2$ and $z_1z_2^2\overline{z}_1^2\overline{z}_2^3$ are in $\mathcal{H}_{3,5}(\3)$. In previous arguments, such specificity was unnecessary, but we find that for $\heart_b$, the eigenvalues can directly depend on the exact degrees of the $z_k, \overline{z}_k$. To that end, for non-negative integer tuples $p = (p_1, \ldots, p_n)$ and $q = (q_1, \ldots, q_n)$, we use $\hpqTUPLE(\CC^n)$ to denote the space of all harmonic, homogeneous polynomials where $p_k$ is the degree of $z_k$, and $q_k$ is the degree of $\overline{z}_k$. Then we use $\hpqTUPLE(\Sn)$ to denote the restriction of this space on $\Sn$. For example, now $z_1^2z_2\overline{z}_1^3\overline{z}_2^2\in \mathcal{H}^{\ast}_{(2,1),(3,2)}(\3)$ but $z_1z_2^2\overline{z}_1^2\overline{z}_2^3\in \mathcal{H}^{\ast}_{(1,2),(2,3)}(\3)$.
Note that $\hpqTUPLE(\Sn)$ is a subspace of $\mathcal{H}_{\overline{p},\overline{q}}(\Sn)$, where $\overline{p} = \sum_{i=1}^n p_i$ and $\overline{q} = \sum_{i=1}^n q_i$. Now for certain $\hpqTUPLE(\Sn)$, we have the following result.


\begin{lemma}
Consider two non-negative integer tuples $p = (p_1, \ldots, p_n)$ and $q = (q_1, \ldots, q_n)$. Suppose that for each $1\leq k\leq n$, at least one of $p_k$ or $q_k$ is 0. Then the eigenvalue of $\heart_b$ on $\hpqTUPLE(\Sn)$ is 
\[   p_1 \sum_{k=2}^n q_k + q_1 \sum_{k=2}^n p_k   +   (n-1) q_1  +  \sum_{k=2}^n q_k   . \]
\end{lemma}
\begin{proof}
Take $f \in \hpqTUPLE(\Sn)$, where $p_k =0$ or $q_k = 0$ for each $k$. By linearity, we can inspect the action of each $-M_{1k}\overline{M}_{1k}$ piece of $\heart_b$ on $f$. We have
\begin{align*}
-M_{1k}\overline{M}_{1k} f &= - 
\Big(\overline{z}_1 \frac{\partial}{\partial z_k} - \overline{z}_k \frac{\partial}{\partial z_1} \Big) 
\Big(  z_1 \frac{\partial}{\partial \overline{z}_k} - z_k \frac{\partial}{\partial \overline{z}_1} \Big) 
f \\
&= - \overline{z}_1 \frac{\partial}{\partial z_k} z_1 \frac{\partial}{\partial \overline{z}_k} f
+ \overline{z}_1 \frac{\partial}{\partial z_k} z_k \frac{\partial}{\partial \overline{z}_1} f
+ \overline{z}_k \frac{\partial}{\partial z_1}  z_1 \frac{\partial}{\partial \overline{z}_k} f
- \overline{z}_k \frac{\partial}{\partial z_1} z_k \frac{\partial}{\partial \overline{z}_1} f \\
&= - z_1 \overline{z}_1 \frac{\partial}{\partial z_k}  \frac{\partial}{\partial \overline{z}_k} f 
+ \overline{z}_1 \frac{\partial}{\partial \overline{z}_1}  \frac{\partial}{\partial z_k} z_k f
+ \overline{z}_k  \frac{\partial}{\partial \overline{z}_k} \frac{\partial}{\partial z_1}  z_1 f
- z_k \overline{z}_k \frac{\partial}{\partial z_1} \frac{\partial}{\partial \overline{z}_1} f \\
&= 0 + \overline{z}_1 \frac{\partial}{\partial \overline{z}_1}  \frac{\partial}{\partial z_k} z_k f
+ \overline{z}_k  \frac{\partial}{\partial \overline{z}_k} \frac{\partial}{\partial z_1}  z_1 f
- 0  \\
&= \overline{z}_1 \frac{\partial}{\partial \overline{z}_1} \Big( z_k \frac{\partial}{\partial z_k}  + I \Big) f
+ \overline{z}_k  \frac{\partial}{\partial \overline{z}_k} \Big(  z_1 \frac{\partial}{\partial z_1} + I \Big)f \\
&= \overline{z}_1 \frac{\partial}{\partial \overline{z}_1} z_k \frac{\partial}{\partial z_k}f  
+ \overline{z}_1 \frac{\partial}{\partial \overline{z}_1} f
+ \overline{z}_k  \frac{\partial}{\partial \overline{z}_k}  z_1 \frac{\partial}{\partial z_1} f 
+ \overline{z}_k  \frac{\partial}{\partial \overline{z}_k} f \\
&= q_1 p_k f  + q_1 f + q_k  p_1 f + q_k  f \\
\end{align*}
Thus, we have 
\begin{align*}
\heart_b(f) &= \sum_{k=2}^n -M_{1k}\overline{M}_{1k} f \\
&= \sum_{k=2}^n  (q_1 p_k   + q_1  + q_k  p_1  + q_k ) f \\
&= \Big( \sum_{k=2}^n  q_1 p_k   +   \sum_{k=2}^n q_1  +   \sum_{k=2}^n q_k  p_1  +  \sum_{k=2}^n q_k \Big) f \\
&= \Big(  q_1 \sum_{k=2}^n p_k   +   (n-1) q_1  +   p_1 \sum_{k=2}^n q_k   +  \sum_{k=2}^n q_k \Big) f. \\
\end{align*}
\end{proof}

The above lemma tells us that $z_1^2 z_2 \overline{z}_1^3 \overline{z}_2^2 \in \mathcal{H}^*_{(2,1),(3,2)} (\mathbb{S}^3) $ has eigenvalue $2(2) + 3(1) + (2-1)(3) + (2) = 12$. On the other hand, $z_1 z_2^2 \overline{z}_1^2 \overline{z}_2^3 \in \mathcal{H}^*_{(1,2),(3,2)} (\mathbb{S}^3)$ has eigenvalue $1(2) + 3(2) + (2-1)(3) + (2)=13$. More generally, the lemma tells us that $\HH_{p,0}(\Sn)$ is in the null space of $\heart_b$ for all $p \in \N$. Furthermore, the eigenvalue of $\heart_b$ on $\mathcal{H}^*_{0,q}(\Sn)$ is $(n-1) q_1  + q_2 + \cdots + q_n$. On other $\hpq(\Sn)$ spaces, computational results suggest that we have integer eigenvalues, and matrix representations follow a pattern as well. We will leave the investigation of other eigenvalues to a future study. We invite the interested reader to see other computational results by downloading our code.\footnote{The code can be downloaded at \texttt{https://goo.gl/kBsUzA}.}


\section{Eigenvalues of $\boxbt$ on the Rossi sphere}

Previously in \cite{REU17} authors studied the spectrum of the perturbed Kohn Laplacian $\boxbt$ on the Rossi Sphere. They obtained an upper bound for lowest eigenvalue for $\boxbt$ on each $\HH_k(\3)$. In our project, we look at the asymptotics of the spectrum of the (perturbed) Kohn Laplacian on the Rossi Sphere, in particular the asymptotics of $\lambda_k^{max}$, the maximum eigenvalue of $\boxbt$ on $\HH_k(\3)$.

In \cite{REU17} it is proved that tridiagonal representation results for spaces of homogeneous polynomials of odd degree, $\HH_{2k-1}(\3)$. However, their proof actually works for arbitrary degrees, $\HH_k(\3)$. We restate the steps to construct the tridiagonal matrix representations here, and one can refer to \cite{REU17} for details.  We first recall the definition of differential operators $\LL{}, \lbar{},$ and $\boxbt$ on $L^2(\3)$. 
\begin{definition} We define $\LL{}$ and $\lbar{}$ as
\begin{align*}
\LL{} &= \overline{z}_1 \frac{\partial}{\partial z_2} - \overline{z}_2 \frac{\partial}{\partial z_1}, \\
\lbar{} &= z_1 \frac{\partial}{\partial \overline{z}_2} - z_2 \frac{\partial}{\partial \overline{z}_1}, \\
\boxbt&=-\l_t \frac{1 + |t|^2}{(1 - |t|^2)^2} \lbar_t.
\end{align*}
\end{definition} 
\noindent The motivation for these operators arises from the CR-manifold $(\3, \lt)$, which is not CR-embeddable \cite{Rossi65}. Note that $\lt = \LL{} + \overline{t} \lbar{}$ and $|t| < 1$.


\begin{theorem}{\label{orthBasis}} \cite{REU17} 
Let $\{ f_0, \ldots, f_k \}$ be an orthogonal basis for $\HH_{0, k} (\3)$. Then $\{ \lbar{}^\sigma f_0 , \ldots, \lbar{}^\sigma f_k\}$ is an orthogonal basis for $\HH_{\sigma, k - \sigma} (\3)$. 
\end{theorem}

The proof of Theorem $\ref{orthBasis}$ follows from induction on inner product. The main two steps are the fact that $-\LL{}$ is the adjoint of $\lbar{}$, and that $\LL{}\lbar{}$ scales elements of $\hpq(\3)$ by a constant factor based on their bidegree.

Now one can consider an orthogonal basis $\{ f_0, \ldots, f_k \}$ for $\HH_{0, k} (\3)$ and define the following two subspaces for even $k$: \\
\begin{align*}
V_i &= \mathrm{span} \{ f_i, \lbar{}^2 f_i, \lbar{}^4 f_i, \ldots, \lbar{}^{k-2} f_i, \lbar{}^k f_i \}, \\
W_i &= \mathrm{span} \{ \lbar{} f_i, \lbar{}^3 f_i, \lbar{}^5 f_i, \ldots, \lbar{}^{k-3} f_i, \lbar{}^{k-1} f_i \},
\end{align*}
and similarly for odd $k$:
\begin{align*}
V_i &= \mathrm{span} \{ f_i, \lbar{}^2 f_i, \lbar{}^4 f_i, \ldots, \lbar{}^{k-3} f_i, \lbar{}^{k-1} f_i \}, \\
W_i &= \mathrm{span} \{ \lbar{} f_i, \lbar{}^3 f_i, \lbar{}^5 f_i, \ldots, \lbar{}^{k-2} f_i, \lbar{}^{k} f_i \}.
\end{align*}
The motivation to define such spaces follows by inspecting the expanded form of $\boxbt$, which is equal to $\LL{} \lbar{} + \lbar{} \LL{} + \LL{}^2 + \lbar{}^2$ up to constants. Previous work has shown that $ \LL{} \lbar{}, \lbar{} \LL{}$ scales elements of $\hpq(\3)$ by a constant factor; and the action of $\LL{}^2, \lbar{}^2$ suggests that invariant subspaces will involve basis elements that differ by $2j$ applications of $\lbar{}$. Indeed, it is shown in \cite{REU17} that $\boxbt$ is invariant on $V_i$ and $W_i$. On these finite dimensional invariant subspaces one can obtain a matrix representation for the second order operator $\boxbt$.

\begin{theorem}{\label{matrixRep1}}\cite{REU17}
The matrix representation of $\boxbt$, $m(\square_b^t)$, on $V_i, W_i \subset \HH_k(\3)$ is 

\[ h \begin{pmatrix}
d_1 & u_1 & & & \\
-\overline{t} & d_2 & u_2 & & \\
& -\overline{t} & d_3 & \ddots & \\
& & \ddots & \ddots & u_{k-1} \\
& & & -\overline{t} & d_k
\end{pmatrix} \]
where $h$ is a constant and
on $V_i$, $u_j = -4t \cdot (j)(2j-1)(k-j)(2k-1-2j)$ and $d_j = (2j-1)(2k+1-2j) + 4|t|^2(j-1)(k+1-j)$;
on $W_i$, $u_j = -4t \cdot (j)(2j+1)(k-j)(2k-1-2j)$ and $d_j = 4(j)(k-j) + |t|^2 (2j-1)(2k+1-2j)$. Moreover, the matrix above is similar to
\[ B = \begin{pmatrix}
d_1 & c_1 & & & \\
c_1 & d_2 & c_2 & & \\
& c_2 & d_3 & \ddots & \\
& & \ddots & \ddots & c_{k-1} \\
& & & c_{k-1} & d_k 
\end{pmatrix} \]
where $c_j = (-\overline{t} \cdot u_j)^{1/2} = |t| \; \sqrt[]{-u_j/t}$.
\end{theorem}


After recalling these results, we also introduce the numerical range of a matrix.
\begin{definition}
Given a $n\times n$ square matrix $A$, we define its numerical range $W(A) = \{  \langle Ax, x \rangle; \; x\in \Cn , ||x|| = 1 \} $.
\end{definition}

Also recall that $\lambda_k^{\max}$ denotes the maximum eigenvalue of $m(\square_b^t)$ on $\mathcal{H}_k(\3)$. We first prove the following lower bound.

\begin{lemma}{\label{evalLB}}
There exists a real constant $c > 0$ so that $ \frac{1}{c} k^2 \leq \lambda_k^{\max}$, i.e. $\lambda_k^{\max}=\Omega(k^2)$ 
\end{lemma}
\begin{proof}
For a square matrix $A$, $\sup W(A)$ is an upper bound for the eigenvalues of $A$. Furthermore, if $A$ is Hermitian  then the maximum eigenvalue equals $\sup W(A)$. 

Let $A = m(\boxbt) $ on $W_i$. By the above discussion, since $A$ is similar to a Hermitian matrix $B$, it suffices to show that $\sup W(B) = \Omega(k^2)$. 

Fix $x = e_{k/2}$ for $k$ even, and $x = e_{(k+1)/2}$ for $k$ odd. Since $\langle Be_i, e_j \rangle = a_{ij}'$, by the above matrix representation we have that for $k$ even,
\begin{align*}
\langle B e_{k/2}, e_{k/2} \rangle &= B_{k/2, k/2} \\
&= d_{\frac{k}{2}} \\
&= 4 \Big(\frac{k}{2} \Big) \Big(k- \frac{k}{2} \Big) + |t|^2 \Big(2 \frac{k}{2}-1 \Big) \Big(2k+1-2  \frac{k}{2} \Big) \\
&= k^2 + |t|^2 (k-1)(k+1) \\
&= \Omega(k^2).
\end{align*}
A similar result follows for $k$ odd. Now since $\langle B e_{k/2}, e_{k/2} \rangle \in W(B)$, we have $\sup W(B) = \Omega(k^2)$.
\end{proof}

For the lower bound we invoke Ger\u{s}gorin's circle theorem.

\begin{theorem}{\label{gersh}} \cite{gersh}
Suppose $A$ is a complex square matrix, and $R_i$ is the sum of the absolute values of the off-diagonal entries in the $i^{th}$ row. Then every eigenvalue of $A$ must lie within one of the closed discs $D(a_{ii},R_i) \subset \mathbb \CC$. 
\end{theorem}

Recall that $m(\boxbt)$ on $V_i, W_i$ is similar to the real symmetric matrix $B$. Since $B$ is Hermitian, Theorem $\ref{gersh}$ will give us interval bounds on the real line. Furthermore, the tridiagonal structure of $B$ makes these bounds tight.


\begin{lemma}{\label{evalUB}}
There exists a real constant $c > 0$ so that $\lambda_k^{\max} \leq ck^2$, i.e. $\lambda_k^{\max} = O(k^2)$.
\end{lemma}
\begin{proof}
Applying Theorem \ref{gersh} on $B$, we have that 
\[ D(b_{ii}, R_i) =\Big( d_i -(c_{i-1} + c_i),d_i +(c_{i-1} + c_i) \Big), \]
since the $i^{th}$ row of $B$ has only two off-diagonal entries, $c_{i-1}$ and $c_i$, both of which are non-negative by Theorem \ref{matrixRep1}. Note that for the extremal cases of the first and last row, the radii of these discs will involve only one off-diagonal entry. Now it suffices to show that an upper bound for $M_i = d_i +c_{i-1} + c_i$ is $O(k^2)$. By inspection, $c_{i-1},$ and $c_i$ are $O(k^2)$ because $u_{i-1}, u_{i}$ are $O(k^4)$. Since $d_i$ is $O(k^2)$ as well, we have our result. 
\end{proof}

By combining the last two lemmas we obtain the following statement.
\begin{theorem}
There exists a real $c > 0$ so that $\frac{1}{c}k^2 \leq \lambda_k^{\max} \leq ck^2$, i.e. $\lambda_k^{\max} = \Theta(k^2)$.
\end{theorem}

In addition to the asymptotics $\lambda_k^{\max}$, we computed $\lambda_k^{\max}$ explicitly by using \texttt{SymPy}. Similar codes also work to compute the largest eigenvalues of other operators, such as $\heart_b$, on finite dimensional invariant spaces.

Finally we note that, in this section we studied perturbed Kohn Laplacians on $\3$. One can define similar perturbations on higher dimensional spheres and investigate the corresponding spectra. Although in higher dimensions Boutet de Monvel theorem \cite{Monvel75} guarantees embeddibility of strongly pseudoconvex abstract CR manifolds, it would be still worthwhile to compute the distribution of eigenvalues.

\section*{Acknowledgements}
We thank the anonymous referee for constructive comments. This research was conducted at the NSF REU Site (DMS-1659203) in Mathematical Analysis and Applications at the University of Michigan-Dearborn. We would like to thank the National Science Foundation, National Security Agency, and University of Michigan-Dearborn for their support. 


\end{document}